\documentclass[preprint,12pt]{elsarticle}

\usepackage{amssymb}
\usepackage{amsmath}
\usepackage{amsthm}
\newtheorem{theorem}{Theorem}[section]
\newtheorem{lemma}[theorem]{Lemma}
\numberwithin{equation}{section}
\newcommand{\PP}{\mathbf{P}}
\newcommand{\pp}{\mathbf{p}}

\journal{Physica D}

\begin{document}

\begin{frontmatter}


\title{Orthogonal polynomials, Toda lattices and Painlev\'e equations}
\author{Walter Van Assche\fnref{supp}}
\fntext[supp]{Supported by FWO project G.0C9819N and EOS project 30889451.}
\ead{walter.vanassche@kuleuven.be}
\ead[url]{https://wis.kuleuven.be/analyse/members/walter}
\affiliation{organization={Department of Mathematics, KU Leuven},
               addressline={Celestijnenlaan 200B box~2400},
               city={Leuven},
               postcode={BE-3001},
               country={Belgium}}

\begin{abstract}
We give a survey of the connection between orthogonal polynomials, Toda lattices and related lattices, and Painlev\'e equations (discrete and continuous). \end{abstract}


\begin{keyword}
Orthogonal polynomials \sep Toda lattice \sep Painlev\'e equations \sep discrete Painlev\'e equations
\MSC 33C45 \sep 42C05 \sep 33E17 \sep 34M55 \sep 39A13 \sep 37K10
\end{keyword}

\end{frontmatter}


In this survey we give a brief introduction to orthogonal polynomials on the real line and on the unit circle in Section \ref{sec:OP} to fix the notation.
In Section \ref{sec:Toda} we will show how the Toda lattice and related lattice equations can be described in terms of orthogonal polynomials.
Section \ref{sec:DP} deals with discrete Painlev\'e equations for the recurrence coefficients of orthogonal polynomials, and Section \ref{sec:Pain}
deals with Painlev\'e differential equations for these recurrence coefficients. 

\section{Orthogonal polynomials}  \label{sec:OP}
\subsection{Orthogonal polynomials on the real line}
Let $\mu$ be a positive measure on the real line for which all the moments $m_n$ exist. Then the orthonormal polynomials are given by the
orthogonality relations
\[     \int_{\mathbb{R}} p_n(x)p_m(x)\, d\mu(x) = \delta_{m,n}, \]
with $p_n(x) = \gamma_n x^n +\cdots$ and $\gamma_n > 0$.
One of their most remarkable features is that they always satisfy a three term recurrence relation
\begin{equation}  \label{3trr}
   xp_n(x) = a_{n+1}p_{n+1}(x) + b_n p_n(x) + a_n p_{n-1}(x), \qquad  n \geq 0  
\end{equation}
with initial values $p_0(x)=1/\sqrt{m_0}$ and $p_{-1} = 0$. The recurrence coefficients are given by
\[   a_n = \int_{\mathbb{R}} xp_n(x)p_{n-1}(x)\, d\mu(x), \quad   b_n = \int_{\mathbb{R}} xp_n^2(x)\, d\mu(x), \]
and comparing the coefficient of $x^{n+1}$ in \eqref{3trr} one also finds
\begin{equation}  \label{a-gamma}
    a_{n+1} = \frac{\gamma_n}{\gamma_{n+1}}.  
\end{equation}    

The monic orthogonal polynomials are $P_n(x) = p_n(x)/\gamma_n$ and they satisfy the recurrence relation
\begin{equation}  \label{3TRR}
   xP_n(x) = P_{n+1}(x) + b_n P_n(x) + a_n^2 P_{n-1}(x), 
\end{equation}
with $P_0=1$ and $P_{-1}=0$. The square of the norm of the monic polynomial is
\begin{equation}   \label{Pnorm}
   \int_{\mathbb{R}}  P_n^2(x) \, d\mu(x) = \frac{1}{\gamma_n^2}.  
\end{equation}   
Relevant literature for orthogonal polynomials on the real line is Szeg\H{o} \cite{Szego}, the more recent book by Ismail \cite{Ismail} and \cite{AB}.

\subsection{Orthogonal polynomials on the unit circle}
Let $\nu$ be a positive measure on the unit circle $\{ z \in \mathbb{C} : |z| = 1 \}$. The orthonormal polynomials on the unit circle are given by
the orthogonality relations
\[   \int_0^{2\pi} \varphi_n(z) \overline{\varphi_m(z)}\, d\nu(\theta) = \delta_{m,n}, \qquad z=e^{i\theta},  \]
with $\varphi(z) = \kappa_n z^n + \cdots$ and $\kappa_n >0$.
The Szeg\H{o}-Levison recursion relation is
\[      \kappa_{n+1} z\varphi_{n}(z)  =  \kappa_n  \varphi_{n+1}(z) - \varphi_{n+1}(0) \varphi_n^*(z), \]
where $\varphi_n^*(z) = z^n \bar{\varphi}_n(1/z)$  is the reversed polynomial, with $\bar{\varphi}_n$ the polynomial with complex conjugated coefficients.
The reversed polynomials satisfy the orthogonality relations
\[    \int_0^{2\pi} \varphi_n^*(z) z^{-k}\, d\nu(\theta) = 0, \qquad   1 \leq k \leq n.  \]
The monic orthogonal polynomials $\Phi_n(z) = \varphi_n(z)/\kappa_n$ satisfy
\begin{equation}   \label{RecSz}
     z \Phi_n(z) = \Phi_{n+1}(z) + \overline{\alpha_n} \Phi_n^*(z), 
\end{equation}
with recurrence coefficients $\alpha_n = - \overline{\Phi_{n+1}(0)}$ which are nowadays known as Verblunsky coefficients.
The square of the norm of the monic orthogonal polynomial is
\[   \int_0^{2\pi} |\Phi_n(z)|^2 \, d\nu(\theta) = \frac{1}{\kappa_n^2}.  \]
Observe that $|z|=1$ on the unit circle so that $\Phi_n^*(z) = z^n\overline{\Phi_n(z)}$ on the unit circle, hence taking the norm on both sides of \eqref{RecSz} gives
\[    \frac{1}{\kappa_n^2} = \frac{1}{\kappa_{n+1}^2} + \frac{|\alpha_n|^2}{\kappa_n^2}, \]
from which one finds
\begin{equation}   \label{kappa-alpha}
     \frac{\kappa_n^2}{\kappa_{n+1}^2} = 1 - |\alpha_n|^2, 
\end{equation}     
which implies that $|\alpha_n| < 1$ for all $n$. One usually takes $\alpha_{-1}=-1$ which is compatible with $\Phi_0 = 1$.

The recurrence relation \eqref{RecSz} and its $*$-companion can be written in matrix form as
\begin{equation}   \label{Phitransfer}
    \begin{pmatrix} \Phi_{n+1}(z) \\ \Phi_{n+1}^*(z) \end{pmatrix}
     = \begin{pmatrix}  z & -\overline{\alpha_n} \\ -\alpha_n z & 1 \end{pmatrix}
          \begin{pmatrix} \Phi_n(z) \\ \Phi_n^*(z) \end{pmatrix}
\end{equation}          
and the matrix in this expression then serves as a transfer matrix with determinant $z(1-|\alpha_n|^2)$.

Orthogonal polynomials on the unit circle already appear in Szeg\H{o}'s book \cite[Ch. XI]{Szego}. The two books by Simon \cite{Simon} are
strongly recommended. See also Ismail \cite[Ch. 8]{Ismail} and \cite[Ch. 9]{AB}.

\section{Toda lattices}  \label{sec:Toda}
The Toda lattice is a system of differential equations for particles with an exponential interaction
\[   x_n''(t) = e^{x_{n-1}-x_n} - e^{x_n-x_{n+1}}, \qquad n \in \mathbb{Z},  \]
which was introduced by Morikazu Toda in 1967 (Toda \cite{Toda}).
We will consider the semi-infinite system (with $x_{-1}=-\infty$) or the finite system (with $x_{-1}=-\infty$ and $x_{n+1}=+\infty$)
for the connection with orthogonal polynomials on the real line.

The change of variables (Flaschka variables)
\[     a_k^2 = \exp(x_{k-1}-x_k), \quad b_k=-x_k', \]
gives the system of equations
\begin{eqnarray}
      (a_k^2)' &=&  a_k^2(b_k-b_{k-1}), \qquad 1 \leq k \leq n,    \label{Toda-a} \\
      b_k' &=& a_{k+1}^2 - a_k^2, \qquad 0 \leq k \leq n, \quad   \label{Toda-b}
\end{eqnarray}      
with $a_0^2 = 0 = a_{n+1}^2$ for the finite system and $a_0^2=0$ for the semi-infinite system.

\subsection{The Toda lattice}    
The semi-infinite Toda lattice is related to orthogonal polynomials for an exponential modification of a positive measure $\mu$ on the real line
with a factor $e^{xt}$:
\[    \int_{\mathbb{R}} p_n(x,t) p_m(x,t) e^{xt} \, d\mu(x) = \delta_{m,n},  \]
whenever all the moments of the measure $e^{xt}\, d \mu(x)$ exist. The Toda equations express the compatibility between the recurrence
relation \eqref{3TRR} and the dynamics of these polynomials as a function of $t$.

\begin{lemma}
Let $P_n(x,t)$ be the monic orthogonal polynomials for the measure $d\mu_t(x) = e^{xt}\, d\mu(x)$. If all the moments of $\mu_t$ exist,
then
\begin{equation}  \label{dPdt}
  \frac{d}{dt} P_n(x,t) = -a_n^2(t) P_{n-1}(x,t).  
\end{equation}
\end{lemma}

\begin{proof}
Since $P_{n}(x,t) = x^n + \cdots$, the derivative $dP_n/dt$ is a polynomial of degree $n-1$. Differentiating the orthogonality relations
\[   \int_{\mathbb{R}} P_n(x,t) x^k e^{xt}\, d\mu(x) = 0, \qquad 0 \leq k \leq n-1, \]
gives
\[    \int_{\mathbb{R}} \frac{dP_n(x,t)}{dt} x^k e^{xt}\, d\mu(x) + \int_{\mathbb{R}} P_n(x,t) x^{k+1} e^{xt}\, d\mu(x) = 0, \qquad
0 \leq k \leq n-1, \]
hence the orthogonality of the polynomial $P_n$ implies that
\[      \int_{\mathbb{R}} \frac{dP_n(x,t)}{dt} x^k e^{xt}\, d\mu(x)  = 0, \qquad  0 \leq k \leq n-2, \]
so that $dP_n(x,t)/dt$ is indeed  $c_n(t) P_{n-1}(x,t)$ for some $c_n(t)$. This $c_n(t)$ satisfies
\[     \int_{\mathbb{R}} \frac{dP_n(x,t)}{dt} x^{n-1} e^{xt}\, d\mu(x) = c_n(t) \int_{\mathbb{R}} P_{n-1}(x,t) x^{n-1} e^{xt}\, d\mu(x) = \frac{c_n(t)}{\gamma_{n-1}^2} . \]
For $k=n-1$ we have
\[   \int_{\mathbb{R}} \frac{dP_n(x,t)}{dt} x^{n-1} e^{xt}\, d\mu(x) + \int_{\mathbb{R}} P_n(x,t) x^{n} e^{xt}\, d\mu(x) = 0, \]
and since
\[    \int_{\mathbb{R}} P_n(x,t) x^n e^{xt}\, d\mu(x) = \int_{\mathbb{R}} P_n^2(x,t) e^{xt}\, d\mu(x) = \frac{1}{\gamma_n^2}  \]
we find
\[   \int_{\mathbb{R}}  \frac{dP_n(x,t)}{dt} x^{n-1} e^{xt}\, d\mu(x) = - \frac{1}{\gamma_n^2} = \frac{c_n(t)}{\gamma_{n-1}^2},  \]
so that
\[   c_n(t) = - \frac{\gamma_{n-1}^2}{\gamma_n^2} = -a_n^2(t).  \]
\end{proof}

If we now differentiate the recurrence relation \eqref{3TRR} with respect to the variable $t$, then
\begin{multline*} 
     x \frac{d}{dt} P_n(x,t) = \frac{d}{dt} P_{n+1}(x,t) + \frac{db_n(t)}{dt} P_n(x,t) + b_n(t) \frac{d}{dt}P_n(x,t)   \\
      + \frac{da_n^2}{dt} P_{n-1}(x,t)      + a_n^2 \frac{d}{dt} P_{n-1}(x,t).  
\end{multline*}
Now use \eqref{dPdt} to find
\begin{multline*}
    - a_n^2 xP_{n-1}(x,t) = -a_{n+1}^2 P_n(x,t) +    \frac{db_n(t)}{dt} P_n(x,t) - b_n(t) a_n^2 P_{n-1}(x,t)  \\
    +   \frac{da_n^2}{dt} P_{n-1}(x,t)
    - a_n^2 a_{n-1}^2 P_{n-2}(x,t).  
\end{multline*}
For $xP_{n-1}(x,t)$ we use the three term recurrence relation \eqref{3TRR} to find
\begin{multline}  \label{eq:2.4}
    -a_n^2 \bigl( P_n(x,t) + b_{n-1}P_{n-1}(x,t) + a_{n-1}^2 P_{n-2}(x,t) \bigr)  \\
    =    -a_{n+1}^2 P_n(x,t) +    \frac{db_n(t)}{dt} P_n(x,t) - b_n(t) a_n^2 P_{n-1}(x,t)  \\ +   \frac{da_n^2}{dt} P_{n-1}(x,t) 
    - a_n^2 a_{n-1}^2 P_{n-2}(x,t).  
\end{multline}    
Comparing the coefficients of $P_n(x,t)$ in \eqref{eq:2.4} gives
\[    \frac{db_n}{dt} = a_{n+1}^2 - a_n^2,  \]
and the coefficients of $P_{n-1}(x,t)$ give
\[    \frac{da_n^2}{dt} = a_n^2 (b_n-b_{n-1}) .  \]
These are the Toda lattice equations \eqref{Toda-a}--\eqref{Toda-b} for $n \geq 0$ and $a_0^2=0$. The initial conditions $a_n^2(0)$ and $b_n(0)$ are given by the recurrence coefficients of the orthogonal polynomials for the measure $\mu$.
A method for solving the semi-infinite Toda lattice equations is:
\begin{enumerate}
\item Use the initial conditions $a_n^2(0)$ and $b_n(0)$ to determine the orthogonality measure $\mu$ for the orthogonal polynomials $P_n(x,0)$.
\item Consider the modified measure $d\mu_t(x) = e^{xt} \, d\mu(x)$ to describe the dynamics as a function of $t$.
\item Find the recurrence coefficients for the orthogonal polynomials for the measure $d\mu_t(x) = e^{xt} \, d\mu(x)$.
\end{enumerate} 
Step 3 is known as the direct problem for orthogonal polynomials: find the recurrence coefficients if one knows the orthogonality measure.
Step 1 is the inverse problem for orthogonal polynomials: find the orthogonality measure if one knows the recurrence coefficients.
If one looks at the three term recurrence relation as a discrete version of the Schr\"odinger equation, then the recurrence coefficients correspond
to a discrete potential. Finding the potential from spectral data is known as the inverse problem in scattering theory, and hence the direct problem
for orthogonal polynomials corresponds to the inverse problem in discrete scattering.

\subsection{Lax pair for the Toda lattice}
One can consider the recurrence relation \eqref{3TRR} and the dynamics \eqref{dPdt} as the Lax pair for the Toda lattice, since the Toda lattice
equations express the compatibility between those two difference and differential equations. It is however somewhat more practical and handy
to use the corresponding Jacobi operator $J$ which has the matrix representation
\begin{equation}   \label{Jacobi}
  J = \begin{pmatrix}  b_0 & a_1 & 0 & 0 & 0 & 0 & 0 & 0 &\cdots \\
                                       a_1 & b_1 & a_2 & 0 & 0 & 0 & 0  & 0 & \cdots \\
                                       0 & a_2 & b_3 & a_3 & 0 &  0 & 0  & 0 &\cdots\\
                                       \vdots & & \ddots & \ddots & \ddots & & & &\vdots   \\
                                       0 & \cdots & 0 & a_{n-1} & b_{n-1} & a_{n} & 0 & 0 & \cdots\\
                                       0  & \cdots & 0 & 0 & a_{n} & b_{n} & a_{n+1} & 0 &\cdots\\
                                       0 & \cdots & 0 & 0 & 0 & a_{n+1} & b_{n+1} & a_{n+2}  &  \\
                                       \vdots & \cdots & \vdots & \vdots & \vdots  & & \ddots & \ddots & \ddots
               \end{pmatrix}  .
\end{equation}               
Let $\Lambda$ be the diagonal matrix $\textup{diag}(\gamma_0,\gamma_1,\gamma_2,\ldots,\gamma_n, \ldots)$, then
the recurrence relation \eqref{3trr} can  be written as
\begin{equation}  \label{JP}
     J  \Lambda \PP = x \Lambda \PP  
\end{equation}     
and \eqref{dPdt} becomes
\begin{equation}   \label{dPA}
       \frac{d}{dt} \PP = - \Lambda^{-1} J_- \Lambda \PP,  
\end{equation}       
where $J_-$ is the lower triangular part of $J$, and 
\[   \PP = \begin{pmatrix} P_0(x,t) \\ P_1(x,t) \\ P_2(x,t) \\ \vdots \\ P_n(x,t) \\ \vdots  \end{pmatrix} .  \]

\begin{theorem}   \label{thm:Toda-Lax}
The Toda lattice equations are
\[    J' = [J,A], \]
where $[J,A]=JA-AJ$ and $A= \frac12 (J_ - - J_+)$, with $J_-$ the lower triangular part of $J$ and $J_+$ the upper triangular part of $J$.
\end{theorem}              

\begin{proof}
If we use the orthonormal polynomials $\pp = \Lambda \PP$, then \eqref{JP} becomes $J\pp = x\pp$. Differentiating with respect to $t$ gives
\[    J' \pp + J \pp' = x\pp'.  \]
Equation \eqref{dPA} in terms of orthonormal polynomials is 
\[   \pp' = L \pp - J_- \pp, \]
with $L = -\Lambda (\Lambda^{-1})' =  \textup{diag}(\gamma_0'/\gamma_0,\gamma_1'/\gamma_1,\ldots,\gamma_n'/\gamma_n,\ldots)$.
We thus find
\[   J' \pp + J (L \pp - J_- \pp) = x (L \pp - J_- \pp).  \]
The term $x \pp$ can be replaced by $J \pp$ and one has
\[    J' \pp + J (L \pp - J_- \pp) =  L J \pp - J_- J \pp .  \]
The orthogonality of the orthonormal polynomials implies
\[    \int_{\mathbb{R}} \pp \pp^T e^{xt} \, d\mu(x) = I  \]
where $I$ is the identity matrix, hence we get
\[     J' =  -JL + J J_- + L J - J_- J .  \]
The matrix $J$ is symmetric, so taking the transpose gives 
\[    J' = -LJ + J_+ J + JL - J J_+ .  \]
Summing both equations gives
\[   2 J' = J ( J_- - J_+) - (J_- - J_+) J, \]
which is the desired form of the Toda equations.
\end{proof}

The pair of infinite matrices $(J,A)$ is known as the Lax pair for the Toda equations. They imply that the Toda lattice is an integrable system.

\subsection{The Toda hierarchy}
The orthogonal polynomials for the measure $e^{x^kt}\, d\mu(x)$, with $k$ a positive integer, 
\[    \int_{\mathbb{R}} p_n(x,t) p_m(x,t) e^{x^kt}\, d\mu(x) = \delta_{m,n}, \]
are related to the Toda hierarchy. The Toda lattice corresponds to $k=1$. The dynamics of these polynomials as a function of the parameter $t$ is:

\begin{lemma} 
Let $P_n(x,t)$ be the monic orthogonal polynomials for the measure $d\mu_t(x) = e^{x^kt}\, d\mu(x)$. If all the moments of $\mu_t$ exist, then
\begin{equation}  \label{dPdt-k}
   \frac{d}{dt} P_n(x,t) = - \sum_{j=1}^k (J^k)_{n,n-j} \frac{\gamma_{n-j}}{\gamma_n} P_{n-j}, 
\end{equation}   
where $J$ is the Jacobi matrix \eqref{Jacobi}.
\end{lemma} 

\begin{proof}
Clearly the polynomial $dP_n(x,t)/dt$ is a polynomial of degree $n-1$. We can express it in terms of the monic orthogonal polynomials as
\[   \frac{d}{dt} P_n(x,t) = \sum_{j=1}^n c_j P_{n-j}(x,t),  \]
and the coefficients $c_j$ satisfy
\[    c_j \int_{\mathbb{R}} P_{n-j}^2(x,t) e^{x^kt}\, d\mu(x) = \int_{\mathbb{R}} \frac{dP_n(x,t)}{dt} P_{n-j}(x,t) e^{x^kt}\, d\mu(x).  \]
One has
\[     \int_{\mathbb{R}} P_{n-j}^2(x,t) e^{x^kt}\, d\mu(x) = \frac{1}{\gamma_{n-j}^2}, \]
and if we differentiate
\[   \int_{\mathbb{R}} P_n(x,t) P_{n-j}(x,t) e^{x^kt}\, d\mu(x) = 0, \qquad 1 \leq j \leq n  \]
with respect to $t$, then the orthogonality implies
\[    \int_{\mathbb{R}} \frac{dP_n(x,t)}{dt} P_{n-j}(x,t) e^{x^kt}\, d\mu(x) = -  \int_{\mathbb{R}} x^k P_n(x,t) P_{n-j}(x,t) e^{x^kt}\, d\mu(x). \]
The last integral is zero whenever $j > k$, so we only need the terms $1 \leq j \leq k$. The Jacobi matrix has the property
\[     (J)_{m,n} = \int_{\mathbb{R}} x p_m(x,t) p_n(x,t) e^{x^kt}\, d\mu(x), \]
and in general
\[   (J^k)_{m,n} = \int_{\mathbb{R}} x^k p_m(x,t)p_n(x,t) e^{x^kt}\, d\mu(x).  \]
If we use $p_n(x,t) = \gamma_n P_n(x,t)$, then we find
\[    \int_{\mathbb{R}} x^k P_n(x,t) P_{n-j}(x,t) e^{x^kt}\, d\mu(x) = \frac{(J^k)_{n,n-j}}{\gamma_n\gamma_{n-j}}, \]
so that
\[  c_j = - (J^k)_{n,n-j} \frac{\gamma_{n-j}}{\gamma_n}, \]
from which \eqref{dPdt-k} follows.
\end{proof}

The expression \eqref{dPdt-k} can be written in matrix form as
\[   \Lambda \PP' = - (J^k)_- \Lambda \PP, \]
where $(J^k)_-$ is the lower triangular part of $J^k$. This is equivalent with
\[      \pp' = L \pp  - (J^k)_- \pp, \]
with $L$ as in the proof of Theorem \ref{thm:Toda-Lax}.
Following the same calculations as before, we then find:

\begin{theorem}   \label{thm:Toda-k}
The Toda hierachy corresponds to the differential equations
\[   J' = [J,\frac12 \bigl((J^k)_- - (J^k)_+\bigr)] .   \]
\end{theorem}

As a special case, one can start with a symmetric measure $\mu$, so that $b_n=0$ for all $n \geq 0$. When $k=2$ we then have
\[    (J^2)_{n,n} = a_n^2+a_{n+1}^2, \quad  (J^2)_{n,n-2} = a_n a_{n-1},  \quad  (J^2)_{n,n+2} = a_{n+1}a_{n+2} , \]
and $(J^2)_{n,m} = 0$ elsewhere, so that 
\[      a_n' = \frac12 a_n (a_{n+1}^2 - a_{n-1}^2),  \qquad  n \geq 1, \]
with $a_0=0$, or
\[    (a_n^2)' = a_n^2 (a_{n+1}^2 - a_{n-1}^2) ,  \qquad n \geq 1, \]
and this is known as the Langmuir lattice or the Kac-van Moerbeke equations \cite[\S 2.4]{WVA}.

\subsection{Ablowitz-Ladik lattice}
The Ablowitz-Ladik lattice (or the Schur flow) is related to orthogonal polynomials on the unit circle with an exponential modification of a positive
measure $\nu$ on the unit circle with a factor $e^{t\cos \theta}$:
\[    \int_0^{2\pi} \varphi_n(z,t) \overline{\varphi_m(z,t)} e^{t\cos \theta}\, d\nu(\theta) = \delta_{m,n}.  \]
The Ablowitz-Ladik equations express the compatibility between the recurrence relation \eqref{RecSz} and the dynamics of the orthogonal polynomials
as a function of $t$. See Golinskii \cite{Gol2006} and Nenciu \cite{Nenciu} for a detailed analysis.

\begin{lemma}
Let $\Phi_n(z,t)$ be the monic orthogonal polynomials for the measure $d\nu_t(\theta) = e^{t\cos \theta}\, d\nu(\theta)$ on the unit circle.
Then
\begin{equation}  \label{dPhidt}
     \frac{d}{dt} \Phi_n(z,t) = - \frac{\kappa_{n-1}^2}{2 \kappa_{n}^2} \left( \Phi_{n-1}(z,t) + \overline{\alpha_n} \Phi_{n-1}^*(z,t) \right). 
\end{equation}     
\end{lemma}

\begin{proof}
Differentiating the orthogonality relations
\[   \int_0^{2\pi} \Phi_n(z,t) z^{-k} e^{t\cos \theta} \, d\nu(\theta)  = 0, \qquad 0 \leq k \leq n-1   \]
with respect to $t$ gives for $0 \leq k \leq n-1$
\[     \int_0^{2\pi} \frac{d}{dt} \Phi_n(z,t) z^{-k} e^{t\cos \theta} \, d\nu(\theta) +  \int_0^{2\pi} \Phi_n(z,t) z^{-k} \frac{z+1/z}{2} e^{t\cos \theta} \, d\nu(\theta) = 0. \]
The orthogonality relations for $\Phi_n$ then imply
\[    \int_0^{2\pi} \frac{d}{dt} \Phi_n(z,t) z^{-k} e^{t\cos \theta} \, d\nu(\theta) = 0, \qquad 1 \leq k \leq n-2.  \]
The polynomial $d\Phi_n(z,t)/dt$ is of degree $n-1$, and the above orthogonality relations imply that
\[   \frac{d}{dt} \Phi_n(z,t) = c_1(t) \Phi_{n-1}(z,t) + c_2 \Phi_{n-1}^*(z,t).  \]
The coefficients $c_1$ and $c_2$ can be obtained from
\[        c_1 \int_0^{2\pi} \Phi_{n-1}(z,t) z^{-n} e^{t\cos \theta}\, d\nu(\theta)
    = \int_0^{2\pi} \frac{d\Phi_n(z,t)}{dt} z^{-n} e^{t\cos\theta}\, d\nu(\theta), \]
and
\[        c_2 \int_0^{2\pi} \Phi_{n-1}^*(z,t) e^{t\cos \theta}\, d\nu(\theta)
    =            \int_0^{2\pi} \frac{d\Phi_n(z,t)}{dt} e^{t\cos\theta}\, d\nu(\theta).\]
One has
\[     \int_0^{2\pi}  \Phi_{n-1}(z,t) z^{-n} e^{t\cos \theta}\, d\nu(\theta) 
      =   \int_0^{2\pi}  |\Phi_{n-1}^*(z,t)|^2 e^{t\cos \theta}\, d\nu(\theta) = \frac{1}{\kappa_{n-1}^2}, \]
and
\[   \int_0^{2\pi} \Phi_{n-1}^*(z,t) e^{t\cos \theta}\, d\nu(\theta) 
= \int_0^{2\pi}  z^n \overline{\Phi_n(z,t)} e^{t\cos \theta} \, d\nu(\theta) =\frac{1}{\kappa_{n-1}^2}, \]
and one also has
\begin{align*}  
  \int_0^{2\pi} \frac{d}{dt} \Phi_n(z,t) z^{-n+1} e^{t\cos \theta} \, d\nu(\theta) 
 &= - \int_0^{2\pi} \Phi_n(z,t) z^{-n+1} \frac{z+1/z}{2} e^{t\cos \theta} \, d\nu(\theta) \\ 
 &= - \frac{1}{2\kappa_n^2}, 
\end{align*} 
and
\begin{align*}
    \int_0^{2\pi} \frac{d}{dt} \Phi_n(z,t)  e^{t\cos \theta} \, d\nu(\theta) 
    &= - \int_0^{2\pi} \Phi_n(z,t) \frac{z+1/z}{2} e^{t\cos \theta} \, d\nu(\theta) \\
   &=  - \frac12 \int_0^{2\pi}  z\Phi_n(z) e^{t\cos \theta}\, d\nu(\theta). 
\end{align*} 
By using the recurrence relation \eqref{RecSz} we find
 \[  \int_0^{2\pi}  z\Phi_n(z) e^{t\cos \theta}\, d\nu(\theta) = \overline{\alpha_n} \int_0^{2\pi} \Phi_n^*(z,t) e^{t\cos \theta} \, d\nu(\theta)
    =  \frac{\overline{\alpha_n}}{\kappa_n^2}.  \]
Hence    
\[    c_1(t) =  - \frac{\kappa_{n-1}^2}{2\kappa_n^2}, \quad    c_2(t) = - \overline{\alpha_n} \frac{\kappa_{n-1}^2}{2\kappa_n^2} , \]
from which the result follows.
\end{proof}

The expression \eqref{dPhidt} and its $*$-companion can be written in matrix form as
\begin{equation}  \label{dPhidtmatrix}
    \frac{d}{dt} \begin{pmatrix}  \Phi_n(z,t) \\ \Phi_n^*(z,t) \end{pmatrix}
        = - \frac{\kappa_{n-1}^2}{\kappa_n^2} \begin{pmatrix} 1 & \overline{\alpha_n(t)} \\ z \alpha_n(t) & z \end{pmatrix}
           \begin{pmatrix}  \Phi_{n-1}(z,t) \\ \Phi_{n-1}^*(z,t) \end{pmatrix},
\end{equation}
where the matrix is  $(1-|\alpha_n|^2)$ times the inverse of the transfer matrix in \eqref{Phitransfer}.
The compatibility between the recurrence relation \eqref{Phitransfer} and the dynamics \eqref{dPhidtmatrix} is as follows: differentiate
\eqref{Phitransfer} to find 
\begin{multline*}   \frac{d}{dt}  \begin{pmatrix}  \Phi_{n+1}(z,t) \\ \Phi_{n+1}^*(z,t) \end{pmatrix}             
    = \begin{pmatrix} 0 & \overline{\alpha_n'(t)} \\ - z \alpha_n'(t) & 0 \end{pmatrix}
      \begin{pmatrix}  \Phi_n(z,t) \\ \Phi_n^*(z,t) \end{pmatrix}  \\
       + \begin{pmatrix} z & -\overline{\alpha_n(t)} \\ -z\alpha_n(t) & 1 \end{pmatrix}  
       \frac{d}{dt}  \begin{pmatrix} \Phi_n(z,t) \\ \Phi_n^*(z,t)  \end{pmatrix}  
\end{multline*}       
 and using \eqref{dPhidtmatrix} this gives 
 \begin{multline*}
     - \frac{\kappa_n^2}{2\kappa_{n+1}^2}   \begin{pmatrix}  1 & \overline{\alpha_{n+1}(t)} \\ z \alpha_{n+1}(t) & z \end{pmatrix}
        \begin{pmatrix} \Phi_n(z,t) \\ \Phi_n^*(z,t)  \end{pmatrix}
        =  \begin{pmatrix} 0 & \overline{\alpha_n'(t)} \\ - z \alpha_n'(t) & 0 \end{pmatrix}
      \begin{pmatrix}  \Phi_n(z,t) \\ \Phi_n^*(z,t) \end{pmatrix} \\
             - \frac{\kappa_{n-1}^2}{2\kappa_n^2} z (1-|\alpha_n|^2) \begin{pmatrix} \Phi_{n-1}(z,t) \\ \Phi_{n-1}^*(z,t) \end{pmatrix}.  
\end{multline*}
Multiply both sides by the transfer matrix in \eqref{Phitransfer} with $n \mapsto n-1$ to get
 \begin{multline*}
     - \frac{\kappa_n^2}{2\kappa_{n+1}^2}   \begin{pmatrix} z & - \overline{\alpha_{n-1}(t)} \\ -z \alpha_{n-1}(t) & 1 \end{pmatrix}
          \begin{pmatrix}  1 & \overline{\alpha_{n+1}(t)} \\ z \alpha_{n+1}(t) & z \end{pmatrix}
        \begin{pmatrix} \Phi_n(z,t) \\ \Phi_n^*(z,t)  \end{pmatrix}  \\
        =   \begin{pmatrix} z & - \overline{\alpha_{n-1}(t)} \\ -z \alpha_{n-1}(t) & 1 \end{pmatrix}
        \begin{pmatrix} 0 & \overline{\alpha_n'(t)} \\ - z \alpha_n'(t) & 0 \end{pmatrix}
      \begin{pmatrix}  \Phi_n(z,t) \\ \Phi_n^*(z,t) \end{pmatrix} \\
             - \frac{\kappa_{n-1}^2}{2\kappa_n^2} z (1-|\alpha_n|^2) \begin{pmatrix} \Phi_{n}(z,t) \\ \Phi_{n}^*(z,t) \end{pmatrix}.  
\end{multline*}
If we work out the matrix products, then this identity is true if and only if
\begin{multline*}  - \frac{\kappa_n^2}{2\kappa_{n+1}^2} \begin{pmatrix} 1 - \overline{\alpha_{n-1}}\alpha_{n+1} & \overline{\alpha_{n+1}}-\overline{\alpha_{n-1}} \\
   -\alpha_{n-1}+\alpha_{n+1} & 1 - \alpha_{n-1}\overline{\alpha_{n+1}} \end{pmatrix} \\
    = \begin{pmatrix} \alpha_n' \overline{\alpha_{n-1}} & - \overline{\alpha_n'} \\ - \alpha_n' & \overline{\alpha_n'} \alpha_{n-1} \end{pmatrix}
     - \frac{\kappa_{n-1}^2}{2\kappa_n^2} (1-|\alpha_n|^2) \begin{pmatrix} 1 & 0 \\ 0 & 1  \end{pmatrix}  .
\end{multline*}
If one uses \eqref{kappa-alpha} then one gets the equation
\begin{equation}  \label{Ablowitz-Ladik}
     \alpha_n'(t) = \frac12 (1-|\alpha_n|^2) (\alpha_{n+1}-\alpha_{n-1}) , \qquad n=0,1,2,\ldots 
\end{equation}
which are the equations for the Ablowitz-Ladik lattice (or the Schur flow).

\subsection{Lax pair for the Ablowitz-Ladik lattice}

For orthogonal polynomials on the unit circle, there are two matrix representations that replace the Jacobi matrix: an infinite Hessenberg matrix
known as the GGT matrix (Geronimus-Gragg-Teplyaev) and the CMV matrix (Cantero-Moral-Vel\'azquez \cite{CMV}), see \cite[Chapter 4]{Simon}. 
The CMV matrix is the most convenient for a Lax pair formulation. Even though CMV is named after the authors of \cite{CMV}, the
matrix representation was known earlier in numerical linear algebra, where Ammar, Gragg and Reichel and later Bunse-Gerstner and Elsner
obtained essentially the same results (see Watkins \cite{Watkins} for a survey of their work, and Simon \cite{SimonCMV}).
The CMV matrix appears if one orthogonalizes $1, z, z^{-1}, z^2, z^{-2}, z^3, z^{-3}, \cdots$. Recall that one needs all integer power of $z$ 
for completeness on the unit circle. The orthogonal basis that results is $\{ \chi_n, n=0,1,2,\ldots \}$ with
\[    \chi_{2n}(z) = z^{-n} \varphi_{2n}^*(z), \qquad  \chi_{2n+1}(z) = z^{-n} \varphi_{2n}(z), \qquad   n \geq 0.  \]
The CMV matrix $C = \bigl(c_{m,n}\bigr)_{m,n=0}^\infty$ then has the entries
\[       c_{m,n} =      \int_{0}^{2\pi}  z \chi_n(z) \overline{\chi_m(z)} \, d\nu(\theta).   \]
It is a pentadiagonal infinite matrix, and it has a nice factorization $C=LM$, where
\[   L = \begin{pmatrix} \Theta_0 & 0 & 0 & 0 &\cdots \\
                                    0 & \Theta_2 & 0 & 0 & \cdots \\
                                    0 & 0 & \Theta_4 & 0 & \cdots \\
                                    \vdots & & & \ddots & \vdots \end{pmatrix},
       \quad
      M = \begin{pmatrix} 1 & 0 & 0 & 0 &\cdots \\
                                    0 & \Theta_1 & 0 & 0 & \cdots \\
                                    0 & 0 & \Theta_3 & 0 & \cdots \\
                                    \vdots & & & \ddots & \vdots \end{pmatrix},      \]
where $\Theta_n$ are the $2\times 2$ matrices                                                         
\[   \Theta_n = \begin{pmatrix} \overline{\alpha_n} & \rho_n \\  \rho_n & - \alpha_n  \end{pmatrix}, \qquad \rho_n = \sqrt{1-|\alpha_n|^2}.  \]

\begin{theorem}
Let $\varphi_n(z,t)$ be the orthonomal polynomials for the measure $d\nu_t(\theta) = e^{t\cos\theta} \, d\nu(\theta)$ on the unit circle,
and $C(t)$ the corresponding CMV matrix. Then the Ablowitz-Ladik lattice (Schur flow) has the Lax pair representation
\[     \frac{d}{dt} C(t) = \frac12 [B,C], \]
with 
\begin{align*}  B &= \frac{(C+C^*)_+-(C+C^*)_-}{2} \\
                          &= \frac12 \begin{pmatrix} 0 & \rho_0\overline{\Delta}_0 & \rho_0\rho_1 & &  \\
                                                  -\rho_0\Delta_0 & 0 & \rho_1\Delta_1 & \rho_1\rho_2 &  \\
                                                  -\rho_0\rho_1 & -\rho_1\overline{\Delta}_1 & 0 & \rho_2 \overline{\Delta}_2 & \rho_2\rho_3  \\
                                                         & \ddots & \ddots & \ddots &\ddots & \ddots 
                              \end{pmatrix} = -B^*,  
\end{align*}                                                   
with $\Delta_n = \alpha_{n+1}-\alpha_{n-1}$ and $\alpha_{-1}=-1$.     
\end{theorem}
 
 See Golinskii \cite[Thm. 1]{Gol2006}, \cite[Thm. 9.8.9]{AB}, Simon \cite[\S 42 in Part 2]{Simon}, Nenciu \cite{Nenciu} for a proof of this theorem.

\section{Discrete Painlev\'e equation}  \label{sec:DP}

So far we have looked at orthogonal polynomials $p_n(x,t)$ for a measure $d\mu_t(x) = e^{xt}\, d\mu(x)$ on the real line as functions of two variables
$x$ and $t$. However the degree $n$ is also an important parameter and the three term recurrence relation \eqref{3trr} gives the dynamics in this discrete
variable. So one must really consider $p_n(x,t)$ as a functions of three variables $n,x,t$, with $n$ a variable taking values in the discrete set $\mathbb{N}$.
Of course, one can also consider the monic polynomials $P_n(x,t)$ and the three term recurrence relation \eqref{3TRR}. 
Compatibility between the recurrence relation (which acts as a difference equation for the variable $n$) and the behavior in the variable $x$ (differential or difference equation in $x$) will give difference equations for the coefficients $a_n$ and $b_n$ in the recurrence relation \eqref{3trr} or \eqref{3TRR}.

One of the earliest examples are the orthogonal polynomials for the weight $w(x,t)= e^{-x^4+tx^2}$ on $\mathbb{R}$.
In 1976 G\'eza Freud found that $b_n=0$ (because the weight is symmetric around $0$) and
\begin{equation}  \label{dPI}
    4a_n^2 (a_{n+1}^2+a_n^2+a_{n-1}^2 - \frac{t}{2}) = n, 
\end{equation}    
for the case $t=0$, but Shohat already had this recurrence in 1939 but didn't really do anything with it.
Fokas, Its and Kitaev \cite{FIK} realized that \eqref{dPI} should be thought of as a discrete Painlev\'e equation and called it discrete Painlev\'e I
(d-P$_{\scriptstyle\textup{I}}$).
The equation \eqref{dPI} follows from compatiblity between the recurrence relation 
\[    xp_n(x) = a_{n+1} p_{n+1}(x) + a_n p_{n-1}(x), \]
and the structure relation
\[    p_n'(x) = A_np_{n-1}(x) + C_n p_{n-3}(x),   \]
which follows from the fact that $w'(x) = (-4x^3 + 2tx) w(x)$ and integration by parts (see, e.g., \cite[\S 2.1]{WVA}).

Many other situations were investigated and they all have the same features: the measure is  $d\mu(x) = w(x)\, dx$ and the weight 
$w(x)$ satisfies a simple differential equation
\begin{equation}   \label{Pearson}   
   [\sigma(x) w(x)]'  = \tau(x) w(x), 
\end{equation}   
where $\sigma$ and $\tau$ are polynomials, and this gives a structure relation
\begin{equation}   \label{structure}
  \sigma(x) p_n'(x) = \sum_{k=n-t}^{n+s-1} A_{n,k} p_k(x), 
\end{equation}  
where $s = \deg \sigma$ and $t=\max \{ \deg \tau, \deg \sigma -1 \}$. The differential equation \eqref{Pearson} is known as the Pearson equation
and the corresponding orthogonal polynomials are known as semi-classical orthogonal polynomials. The classical orthogonal polynomials
of Hermite, Laguerre and Jacobi correspond to the case when $s \leq 2$ and $t=1$. Compatibility between the three term recurrence
relation \eqref{3trr} and the structure relation \eqref{structure} then gives a system of equations for the unknown recurrence coefficients $a_n, b_n, A_{n,k}$, and eliminating the $A_{n,k}$ then gives a system of two non-linear recurrence relations for the $a_n, b_n$, which can usually be identified
as some discrete Painlev\'e equation. 

The classification of discrete Painlev\'e equations is more complicated than for the Painlev\'e differential equations, which we will encounter in the
next section. Originally one gave names to discrete Painlev\'e equations in a somewhat chronological order, taking into account the
limiting differential equation that appears as a continuum limit (taking $x=nh$ and $h \to 0$), but nowadays one identifies the discrete Painlev\'e
equation using symmetry and geometry, following the work of Okamoto \cite{Okamoto} and Sakai \cite{Sakai}. A very good survey was given by
Kajiwara, Noumi and Yamada \cite{KNY} and is strongly recommended. So the discrete Painlev\'e equation d-P$_{\scriptstyle\textup{I}}$ in
\eqref{dPI} corresponds to d-P$(A_2^{(1)}/E_6^{(1)})$, with symmetry $A_2^{(1)}$ and surface type $E_6^{(1)}$,
where $A_2^{(1)}$ and $E_6^{(1)}$ are affine Weyl groups,  see \cite[Eq. (8.25)]{KNY}.

The structure relation can be replaced by using a lowering and raising operator for the orthogonal polynomials, as was worked out
by Chen and Ismail (see, e.g., \cite[\S 3.2]{Ismail} or \cite[Ch. 4]{WVA}). If we denote the weight function by $w(x)= \exp(-V(x))$, with $V$ twice differentiable,
and define
\begin{eqnarray*}
  A_n(x) &=& a_n \int \frac{V'(x)-V'(y)}{x-y} \ p_n^2(y) w(y)\, dy, \\
  B_n(x) &=& a_n \int \frac{V'(x)-V'(y)}{x-y}\ p_n(y)p_{n-1}(y) w(y)\, dy ,
\end{eqnarray*}
then the lowering operator $L_{1,n}$  and the raising operator $L_{2,n}$ are given by
\[  L_{1,n} = \frac{d}{dx} + B_n(x), \quad   L_{2,n} = - \frac{d}{dx} + B_n(x) + V'(x), \]
and one has
\[     L_{1,n} p_n(x) = A_n(x) p_{n-1}(x), \quad   L_{2,n} p_{n-1}(x) = \frac{a_n}{a_{n-1}} A_{n-1} p_n(x). \]
Combining both operators gives a second order differential equation
\[    L_{2,n} \Bigl( \frac{1}{A_n(x)} \bigl(L_{1,n} p_n(x) \bigr) \Bigr)  =  \frac{a_n}{a_{n-1}} A_{n-1} p_n(x).  \]
The compatibility between this differential equation (in the variable $x$) and the three term recurrence relation \eqref{3trr} (in the variable $n$)
then gives a system of non-linear difference equations for the recurrence coefficients $a_n,b_n$.

Observe that the functions $A_n$ and $B_n$ not only depend directly on the weight function $w$ and the potential $V$, but also on the polynomials
$p_n$ and $p_{n-1}$, and hence it is not always possible to compute $A_n$ and $B_n$ explicitly without knowing the orthogonal polynomials.
However, one can usually figure out that they are rational functions with known poles and then introduce some unknown parameters. The compatibility
relations will then give relations connecting these parameters with the recurrence coefficients, and eliminating the parameters then give the required
difference equations for the recurrence coefficients $a_n,b_n$. Let us give some examples supplementing the earlier example  $w(x)=\exp(-x^4+tx^2)$
on $(-\infty,\infty)$.
\begin{itemize}
\item (Chen and Its \cite{ChenIts}) if $w(x) = x^\alpha e^{-x-t/x}$ for $x \in [0,\infty)$, then
\[   A_n(x) =  a_n \left(\frac{1}{x} + \frac{c_n}{x^2}\right) , \quad  B_n(x) = -\frac{n}{x} + \frac{d_n}{x^2} . \]
The recurrence coefficients can be expressed in terms of $c_n$ and $d_n$ as
\[  b_n = 2n+\alpha + 1 + c_n, \quad  a_n^2 = n(n+\alpha) +d_n + \sum_{j=0}^{n-1} c_j. \]
If we put $x_n=1/c_n$ and $y_n=d_n$ then
\begin{align*}      
    x_n+x_{n-1} &= \frac{nt-(2n+\alpha)y_n}{y_n(y_n-t)}, \\
    y_n+y_{n+1} &= t - \frac{2n+\alpha+1}{x_n} - \frac{1}{x_n^2}, 
\end{align*}
which can be identified as d-P$((2A_1)^{(1)}/D_6^{(1)})$.    
\item (Basor, Chen, Ehrhardt \cite{BCE}) the Toda evolution of the Jacobi weight is $w(x)=(1-x)^\alpha(1+x)^\beta e^{-tx}$ on $[-1,1]$.
The $A_n$ and $B_n$ are rational functions with poles at $\pm 1$:
\[  A_n(x) = a_n \left( \frac{R_n}{1-x} + \frac{t+R_n}{1+x} \right), \quad  B_n(x) = \frac{r_n}{1-x} + \frac{r_n-n}{1+x}.  \]
The recurrence coefficients can be expressed in terms of $r_n$ and $R_n$ as
\begin{align*}
   tb_n &= 2n+1+\alpha+\beta - t - 2R_n, \\   
   t(t+R_n) a_n^2 &= n(n+\beta) -(2n+\alpha+\beta)r_n - \frac{tr_n(r_n+\alpha)}{R_n},  
\end{align*}   
and $(r_n,R_n)$ satisfy
\begin{align*}
  2t(r_n+r_{n+1}) &= 4R_n^2-2R_n(2n+1+\alpha+\beta-t) -2\alpha t, \\
  \left( \frac{t}{R_n} +1 \right) \left(\frac{t}{R_{n-1}}+1 \right) &= 1+ \frac{n(n+\beta)-(2n+\alpha+\beta)r_n}{r_n(r_n+\alpha)}.
\end{align*}
 An identification as a discrete Painlev\'e equation was not made, but after an appropriate transformation $r_n \to q_n$ and $t/R_n \to f_n-1$ one can recognize d-P$(D_4^{(1)}/D_4^{(1)})$. 
\item (Boelen-Van Assche \cite{BVA}, Clarkson-Jordaan \cite{ClarkJor}) for the modified Laguerre weight
$w(x)=x^\alpha e^{-x^2+tx}$ on $[0,\infty)$ the recurrence coefficients can be written as
\[   2a_n^2 = y_n+n+\alpha/2, \quad 2b_n = t - \frac{\sqrt{2}}{x_n}, \]
and the $(x_n,y_n)$ satisfy the system 
\begin{align*}
     x_nx_{n-1} &= \frac{y_n + n+\alpha/2}{y_n^2-\alpha^2/4}, \\
     y_n+y_{n+1} &= \frac{1}{x_n} \left( \frac{t}{\sqrt{2}} - \frac{1}{x_n} \right).
\end{align*}
Using a rational transformation $(x_n,y_n) \to (p_n,q_n)$, Dzhamay et al.  \cite[Eq. (3.86)]{DFS1} were able to identify the resulting equations
\begin{align*}
   q_n + q_{n+1} &= p_n-t- \frac{n+\alpha+2}{p_n}, \\
   p_n+p_{n-1} &= q_n + t - \frac{n+1}{q_n},
\end{align*}
as the discrete Painlev\'e equations d-P$(A_2^{(1)}/E_6^{(1)})$ .    
\end{itemize}

This also works for discrete orthogonal polynomials, i.e., when the orthogonality is in terms of a discrete measure on a lattice or a $q$-lattice.
Some examples of orthogonal polynomials on the lattice $\mathbb{N}$ are
\begin{itemize}
\item Generalized Charlier polynomials: the orthogonality relations are
\[    \sum_{k=0}^\infty P_n(k)P_m(k) \frac{a^k}{k! (\beta)_k} = 0, \qquad m \neq n,  \]
with $a, \beta >0$. The discrete weight $w_k=a^k/(\beta)_k k!$ satisfies the discrete Pearson equation
\[  w_{k-1}=\frac{k(\beta+k-1)}{a} w_k, \]
and the orthonormal polynomials have the structure relation
\[  p_n(x+1)-p_n(x) = A_n p_{n-1}(x) + B_n p_{n-2}(x).  \]
Compatibility of this structure relation with the three term recurrence relation \eqref{3trr} gives the discrete Painlev\'e equations \cite[Thm. 3.7]{WVA}
\begin{align*}
    (a_{n+1}^2-a)(a_n^2-a) &= ad_n(d_n+\beta-1), \\
     d_{n} + d_{n-1} &= -n-\beta+1+\frac{an}{a_n^2},
\end{align*}
with $d_n=b_n-n$, which is a limiting case of d-P$(D_4^{(1)}/D_4^{(1)})$. Observe that for $a=e^t$ the recurrence coefficients $a_n^2(t)$ and $b_n(t)$ satisfy
the Toda-lattice equations \eqref{Toda-a}--\eqref{Toda-b}.     
\item Generalized Meixner polynomials: the orthogonality relations are
\[   \sum_{k=0}^\infty P_n(k) P_m(k) \frac{(\gamma)_k a^k}{(\beta)_k k!} = 0, \qquad m \neq n, \]
with $a,\beta,\gamma >0$. Here it is more convenient to use the ladder operators to obtain the discrete Painlev\'e equation \cite[\S 5.3]{WVA}
\begin{align*}
  (u_n+v_n)(u_{n+1}+v_n) &= \frac{\gamma-1}{a^2} v_n(v_n-a) \left( v_n - a \frac{\gamma-\beta}{\gamma-1} \right), \\
  (u_n+v_n)(u_n+v_{n-1}) &= \frac{u_n}{u_n-\frac{an}{\gamma-1}} (u_n+a) \left( u_n + a \frac{\gamma-\beta}{\gamma-1} \right), 
\end{align*}
and the recurrence coefficients are given by 
\[   a_n^2 = na - (\gamma-1)u_n, \quad  b_n = n+\gamma-\beta+a - \frac{\gamma-1}{a} v_n.  \]
These discrete Painlev\'e equations correspond to d-P$(E_6^{(1)}/A_2^{(1)})$ in \cite{KNY}.
Again, for $a=e^t$ the recurrence coefficients satisfy the Toda lattice equations \eqref{Toda-a}--\eqref{Toda-b}.
\item Hypergeometric weights: the orthogonality relations are
\[   \sum_{k=0}^\infty P_m(k)P_n(k) \frac{(\alpha)_k(\beta)_k a^k}{(\gamma)_k k!} = 0, \qquad m \neq 0, \]
where $\alpha,\beta,\gamma >0$ and $0 < a < 1$. These orthogonal polynomials were investigated in \cite{FWVA}, who obtained
the discrete Painlev\'e equations, but later on these were simplified in \cite[Thm. 1]{DFS0}. The discrete Painlev\'e equations
are
\begin{align*}
   f_{n+1}f_n &= \frac{g_n(g_n- \gamma+\beta)}{c(g_n+n+\beta-\gamma+1)(g_n+n+\alpha+\beta-\gamma)}, \\
   g_n+g_{n-1} &= -2n+\gamma-\alpha -2\beta+\gamma - \frac{n+\beta}{f_n-1} -\frac{n+\alpha-\gamma}{cf_n-1},
\end{align*}    
which corresponds to d-P$(D_4^{(1)}/D_4^{(1)})$, where
{\small
\[   f_n = \frac{(x_n - \beta)(x_n - \gamma)}{c \bigl((x_n - \alpha)(x_n - \beta) - nx_n - y_n\bigr)}, \]
\[  \rule{-20pt}{0pt} g_n =  - \frac{(x_n - \gamma)\Bigl( \bigl( (x_n - \alpha)(x_n - \beta) - nx_n - y_n \bigr) - t(x_n - \beta)(x_n - \gamma + \beta + n) \Bigr)}{\bigl((x_n - \alpha)(x_n - \beta) - nx_n - y_n\bigr) - t(x_n - \beta)(x_n - \gamma)},  \]}
where the recurrence coefficients are given by
\begin{align*}
   \frac{1 - c}{c} a_n^2 &= y_n + \sum_{k=0}^{n-1} x_k + \frac{n(n + \alpha + \beta - \gamma - 1)}{1-c} , \\
     b_n &= x_n + \frac{n + (n + \alpha + \beta)c - \gamma)}{1-c}.   
\end{align*}
Again the recurrence coefficients satisfy the Toda lattice equations if $c=e^t$.     
\end{itemize}

The prototype example for orthogonal polynomials on the unit circle is the weight function $v(\theta) = e^{t\cos \theta}$, 
where $z=e^{i\theta}$. The orthogonal polynomials
for this weight on the unit circle were first investigated by Periwal and Shevitz \cite{PerShe1990} and they are related to unitary random matrices.
The Verblunsky coefficients for these orthogonal polynomials $\alpha_n(t)$ are real and satisfy \cite[Thm. 3.2]{WVA} 
\[   -\frac{t}{2} (1-\alpha_n^2)(\alpha_{n+1}+\alpha_{n-1}) = (n+1)\alpha_n, \]
which corresponds to discrete Painlev\'e II (d-P$_{\scriptstyle\textup{II}}$). As a function of $t$ they also satisfy the Ablowitz-Ladik equation \eqref{Ablowitz-Ladik}
\[   \alpha_n'(t) = \frac12 (1-\alpha_n^2) (\alpha_{n+1}-\alpha_{n-1}) ,  \]
with initial values $\alpha_n(0)=0$ for $n\geq 0$ and $\alpha_{-1}=-1$, since the orthonormal polynomials for $t=0$ are $\varphi_n(z,0) = z^n$.

\section{Painlev\'e differential equations}  \label{sec:Pain}

The orthogonal polynomials depend on three variables $n,x,t$ and in each of these variables they satisfy differential or difference equations.
The compatibility between these relations for the variables $n$ and $t$ results in the Toda lattice equations or related lattice equations such as
Ablowitz-Ladik (see \S \ref{sec:Toda}). The compatibility between the relations for the variables $n$ and $x$ give discrete Painlev\'e equations 
(see \S \ref{sec:DP}). A combination of the Toda equations and the discrete Painlev\'e equations then gives the compatibility between all three
variables $n,x,t$ of the orthogonal polynomials $P_n(x,t)$. In many cases this results in a second order differential equation that (after some transformation)
can be identified as one of the six Painlev\'e equations. For the examples that we have considered before, one arrives at the following equations.
\begin{itemize}
\item For the weight function $w(x,t) = e^{-x^4+tx^2}$ on $(-\infty,\infty)$ one finds $\textup{P}_{\scriptstyle\textup{IV}}$
\[  x_n''(t) = \frac{(x_n')^2}{2x_n} + \frac{3x_n^3}{2} -tx_n^2 + x_n \left( \frac{n}{4} + \frac{t^2}{8} \right) - \frac{n^2}{32x_n} , \]
where $x_n = a_n^2$ (Magnus \cite{Magnus}). The transformation $2x_n(t) = y(-t/2)$ gives P$_{\scriptstyle \textup{IV}}$ in its usual form.
\item For the weight function $v(\theta) = e^{t\cos \theta}$ on the unit circle, the Verblunsky coefficients $\alpha_n(t)$ satisfy 
\[   \alpha_n'' = - \frac{\alpha_n}{1-\alpha_n^2} (\alpha_n')^2 - \frac{\alpha_n'}{t} - \alpha_n(1-\alpha_n^2) + \frac{(n+1)^2}{t^2} \frac{\alpha_n}{1-\alpha_n^2}, \]
(Periwal and Shevitz \cite{PerShe1990})
and if we put $\alpha_n = (1+y)/(1-y)$ then $y$ satisfies P$_{\scriptstyle\textup{V}}$ but with one of the parameters $\gamma=0$. In that case the
Painlev\'e equation can be transformed to P$_{\scriptstyle\textup{III}}$ \cite[\S 3.1.3]{WVA}. If one puts $w_n=\alpha_n/\alpha_{n-1}$, then 
\[  w_n'' = \frac{(w_n')^2}{w_n} - \frac{w_n'}{t} + \frac{2n}{t} w_n^2 - \frac{2(n+1)}{t} + w_n^3 - \frac{1}{w_n}, \]
which is Painlev\'e III.
\item For the weight function $w(x,t) = x^\alpha e^{-x-t/x}$ on $[0,\infty)$ one has
\[   c_n'' = \frac{(c_n')^2}{c_n} - \frac{c_n'}{t} + (2n+\alpha+1) \frac{c_n^2}{t} + \frac{c_n^3}{t^2} + \frac{\alpha}{s} - \frac{1}{c_n} ,\]
which after a transformation is P$_{\scriptstyle\textup{III}}$ (Chen and Its \cite{ChenIts}).
This shows that the same Painlev\'e equation may appear for the recurrence coefficients
of very different families of orthogonal polynomials. What these families have in common is that the moments can be expressed in terms of Bessel functions
and the solution of the Painlev\'e equations that we want is the special function solution in terms of Bessel functions.
\item For the Jacobi-Toda weight function $w(x,t) = (1-x)^\alpha(1+x)^\beta e^{-tx}$ on $[-1,1]$ one uses the transformation $y(t) = 1+t/R_n(t)$ to find
\begin{multline*}
 y'' = \frac{2y-1}{2y(y-1)} (y')^2 - \frac{y'}{t} + 2(2n+\alpha+\beta+1) \frac{y}{t} - \frac{2y(y+1)}{y-1} \\
   + \frac{(y-1)^2}{t^2} \left( \frac{\alpha^2y}{2}-\frac{\beta^2}{2y} \right), 
\end{multline*}
which corresponds to P$_{\scriptstyle\textup{V}}$ (Basor, Chen and Ehrhardt \cite{BCE}).
\item The modified Laguerre weight $w(x,t) = x^\alpha e^{-x^2+tx}$ on $[0,\infty)$ gives a Painlev\'e IV equation for $x_n$
\[  x_n'' = \frac{3}{2} \frac{(x_n')^2}{x_n} + \frac{\alpha^2}{4} x_n^3 - \frac{x_n}{8} (t^2-4-8n-4\alpha) + \frac{t}{\sqrt{2}} - \frac{3}{4x_n}, \]
and in \cite{DFS1} the identification with the Hamiltonian form of P$_{\scriptstyle\textup{IV}}$ was made.
\item For the generalized Charlier polynomials one puts $x_n(a) = \displaystyle \frac{a}{1-y(a)}$ to find 
\[ y'' = \left( \frac{1}{2y} + \frac{1}{y-1} \right) (y')^2 - \frac{y'}{a} + \frac{(y-1)^2}{a^2} \left( \frac{n^2}{2}y - \frac{(\beta-1)^2}{2y} \right) - \frac{2y}{a}. \]
which is P$_{\scriptstyle\textup{V}}$ but with one of the parameters $\delta=0$. This can be transformed to P$_{\scriptstyle\textup{III}}$, see
\cite[\S 3.2.3]{WVA}.
The moments of the generalized Charlier polynomials are in terms of modified Bessel functions.
\item The generalized Meixner polynomials have a genuine P$_{\scriptstyle\textup{V}}$ with all the parameters different from $0$.
One has  \cite[\S 5.3]{WVA}
\[ y'' =  \left( \frac{1}{2y} + \frac{1}{y-1} \right) (y')^2 - \frac{y'}{a} + \frac{(y-1)^2}{a^2} \left(Ay+ \frac{B}{y}\right) + \frac{Cy}{a} +
  \frac{Dy(y+1)}{y-1}, \]
  with
  \[  A= \frac{(\beta-1)^2}{2}, \quad B= -\frac{n^2}{2}, \quad C = n-\beta+2\gamma, \quad D = -\frac12.  \]
Here $y(a)$ is related to $v_n(a)$ by the rational transformation
\[    v_n(a) = \frac{a\bigl(ay'-(1+\beta-2\gamma)y^2+(1+n-a+\beta-2\gamma)y - n\bigr)}{2(\gamma-1)(y-1)y}.  \]
The moments are in terms of the confluent hypergeometric function $M(a,b,z)$ and the solution we need is the special function solution
of P$_{\scriptstyle\textup{V}}$. 
\item For the hypergeometric weights, Dzhamay et al.\ \cite{DFS1} found the appropriate transformation that leads to Painlev\'e VI:
\begin{multline*}
   f'' = \frac12 \left( \frac{1}{f} + \frac{1}{f-1} + \frac{1}{f-t} \right) (f')^2 - \left( \frac{1}{t} + \frac{1}{t-1} + \frac{1}{f-t} \right) f' \\
      + \frac{f(f-1)(f-t)}{t^2(t-1)^2} \left( A + B \frac{t}{f^2} + C \frac{t-1}{(f-1)^2} + D \frac{t(t-1)}{(f-t)^2} \right), 
\end{multline*}
with $ct=1$ and parameters
{\small
\[  A = \frac{(\alpha-1)^2}{2}, \quad B=- \frac{(\beta-\gamma)^2}{2}, \quad C= \frac{(n+\beta)^2}{2}, \quad D= \frac12  - \frac{(n+\alpha-\gamma)^2}{2}. \]  }    
The moments of these discrete weights are in terms of Gauss hypergeometric functions, and the solution that we need is a special function
solution of P$_{\scriptstyle\textup{VI}}$.
\end{itemize}

\section{Conclusion}
Orthogonal polynomials on the real line and the unit circle for semi-classical weights have a remarkable underlying integrable structure. One needs to
look at orthogonal polynomials $P_n(x,t)$ as functions of three variables $n,x,t$, where $n$ is discrete and denotes the degree, $x$ is the usual variable for the polynomial which may be continuous or discrete, an $t$ is the time variable for the Toda-evolution. The three term recurrence relation for the
orthogonal polynomials plays a crucial role and the corresponding Jacobi operator is part of the Lax pair for the Toda lattice. On the unit circle one uses the CMV matrix instead of the Jacobi matrix. Differential equations or difference equations for the orthogonal polynomials (in the variable $x$) lead to non-linear
difference equations for the recurrence coefficients which can be identified as discrete Painlev\'e equations. This identification is not straightforward but
recently a lot of progress has been made using the geometric theory behind (discrete) Painlev\'e equations (Kajiwara et al. \cite{KNY}, Dzhamay et al. \cite{DFS0,DFS1}). A combination of the discrete Painlev\'e equations and the Toda equations finally leads to Painlev\'e equations for the
recurrence coefficients. In all the cases considered one needs the special function solutions of the Painlev\'e equations, and the special functions
are already visible in the first few moments of the weights under consideration. This simplifies the identification because one knows which special functions
can appear in special function solutions of Painlev\'e equations, see e.g., Clarkson \cite{Clarkson}.

\end{document}